\documentclass[11pt]{amsart}
\usepackage[utf8]{inputenc}
\usepackage{csquotes} 
\usepackage[english]{babel}
\usepackage[margin=1in]{geometry}

\usepackage{amsmath}
\usepackage{amssymb}
\usepackage{amsfonts}
\usepackage{amsthm}
\usepackage{mathtools}

\usepackage[pdfencoding=auto,hidelinks,psdextra]{hyperref}
\usepackage[capitalize,nameinlink]{cleveref}
\usepackage{xcolor}
\hypersetup{colorlinks=false}

\usepackage{appendix}
\theoremstyle{plain}
    \newtheorem{theorem}{Theorem}
    \newtheorem{lemma}[theorem]{Lemma}
    
    \newtheorem{conjecture}[theorem]{Conjecture}
    
    \Crefname{lemma}{Lemma}{Lemmas}
    \Crefname{corollary}{Corollary}{corollaries}
    \Crefname{conjecture}{Conjecture}{conjectures}
    \Crefname{proposition}{Proposition}{Propositions}
\theoremstyle{definition}

    \newtheorem*{notation}{Notation}
    
    \Crefname{definition}{Definition}{Definitions}
    \Crefname{example}{Example}{Examples}
\theoremstyle{remark}

    \Crefname{remark}{Remark}{Remarks}
    \Crefname{claim}{Claim}{Claims}

\numberwithin{equation}{section}
\numberwithin{theorem}{section}

\usepackage{mathrsfs}
\usepackage{indent first}
\usepackage{enumitem}
\usepackage{upgreek}
\usepackage{physics}
\usepackage{stmaryrd}

\DeclareMathOperator{\dist}{dist}
\newcommand{\dg}{^\circ}

\usepackage{float}
\usepackage{graphicx}
\usepackage{asymptote}
\begin{asydef}
  defaultpen(fontsize(10pt));
  usepackage("amsmath");
  usepackage("amssymb");
  void dotd(string s, pair p) {
    dot("$" + s + "$", p, dir(p));
  }
\end{asydef}
\usepackage[labelfont=bf,labelsep=period,justification=raggedright]{caption}
\usepackage{subcaption}

\author{Eyvindur A. Palsson}
\email{\textcolor{blue}{\href{mailto:palsson@vt.edu}{palsson@vt.edu}}}
\address{Department of Mathematics, Virginia Tech, Blacksburg, VA 24061}
\author{Edward Yu}
\email{\textcolor{blue}{\href{mailto:edwardyu@mit.edu}{edwardyu@mit.edu}}}
\address{Department of Mathematics, Massachusetts Institute of Technology, Cambridge, MA, 02139}


\begin{document}
\title{On Optimal Point Sets Determining Distinct Triangles}

\begin{abstract}
    Erd\H os and Fishburn studied the maximum number of points in the plane that span exactly $k$ distances (i.e. the set of pairwise distances between points has cardinality $k$) and classified these configurations, as an inverse problem of the Erd\H os distinct distances problem.
    We consider the analogous problem for triangles. Past work has obtained the optimal sets for one and two distinct triangles in the plane.
    In this paper, we resolve a conjecture that at most six points in the plane can span exactly three distinct triangles, and obtain the hexagon as the unique configuration that achieves this.
    We also provide evidence that optimal sets cannot be on the square lattice in the general case.
\end{abstract}

\maketitle

\section{Introduction}

A famous open problem in combinatorial geometry is Erd\H os' distinct distances problem, which asks for the minimum number of distinct distances between pairs of points in an $n$-point subset of the Euclidean plane. Erd\H os conjectured~\cite{erdos_distinct_distances} that any set of $n$ points spans $\Omega\left(\frac{n}{\sqrt{\log n}}\right)$ distinct distances; this was essentially resolved by Guth and Katz~\cite{guth_katz}, whose work implies a bound of $\Omega\left(\frac{n}{\log n}\right)$.
In higher dimensions $\mathbb R^d \; (d\geq3)$, it is conjectured that the maximum number of distinct distances is asymptotic to $n^{2/d}$, with the current best-known results due to Solymosi and Vu~\cite{solymosi_vu}.

The inverse problem to Erd\H os' distinct distances problem is to determine the maximum number of points in the plane that span only a fixed number of distinct distances. This problem was first studied by Erd\H os and Fishburn \cite{erdos_fishburn}, who found the maximum number of points in a subset of the plane spanning $k$ distances, and also determined all associated extremal configurations, for $k\leq4$. This work was extended by Shinohara~\cite{shinohara}, who determined the maximal $5$-distance sets, and Wei~\cite{wei}, who found the maximum size of $6$-distance sets, though the optimal configurations remain unknown.

In this paper, we consider the analogue of the Erd\H os-Fishburn and Erd\H os distinct-distances problems for triangles. This is a problem previously studied by Epstein et. al.~\cite{distinct_triangles_fishburn}, finding the maximal number of points in a set determining $k$ distinct triangles and characterizing the optimal configurations for $k=1,2$.
Optimal point sets for one and two distinct triangles in higher-dimensional Euclidean spaces have since been characterized~\cite{one_dist_triangle, two_dist_triangle}.

To avoid ambiguity regarding degenerate triangles, we can formalize this as counting the number of triples of distinct points in a point set, up to the equivalence relation defined by rigid plane isometries (hence, the three vertices of the triangle are allowed to be collinear, but must be distinct).
This is equivalent to characterizing a ``triangle'' by its side lengths, up to permutation.

In this paper, we extend the results of \cite{distinct_triangles_fishburn} to the $k=3$ case:
\begin{theorem}\label{thm.3-triangles}
    Any set of points in the plane that span at most $3$ distinct triangles contains at most $6$ points, and the only configuration that achieves this is the regular hexagon.
\end{theorem}

In addition, we consider the density of distinct triangles in the square lattice.
The square lattice gave rise to the original asymptotic conjecture for the Erd\H os distinct distance problems, while trianglular lattices have slightly fewer distinct distances~\cite{lattice_distinct_distances}; Erd\H os and Fishburn~\cite{erdos_fishburn} conjectured that optimal configurations for any sufficiently large number of distinct distances exist on the triangular lattice.

It has been shown in~\cite{distinct_triangles_2d} that an $n$-point set spans at least $\Omega(n^2)$ distinct triangles; the true constant is unknown. Epstein et. al.~\cite{distinct_triangles_fishburn} conjectured that the asymptotic is at least $\frac1{12} n^2$, which is achieved by the regular $n$-gon, but it is unknown if other optimal configurations exist. A natural place to search for such configurations are lattices. However, we show that the square lattice cannot give rise to optimal configurations.

\begin{theorem}\label{thm.triangle-grid}
    The $\sqrt n \times \sqrt n$ square grid contains between $0.1558 n^2 + o(n^2)$ and $0.1875 n^2 + o(n^2)$ distinct triangles.
\end{theorem}

This suggests that the trend where structures with lattice symmetry have fewer distinct distances than structures with rotational symmetry is reversed for the distinct triangles problem: the regular $n$-gon has roughly half the number of distinct triangles as the square grid.

\section{Proof of \texorpdfstring{\Cref{thm.3-triangles}}{Theorem 1.1}}

\begin{notation}
    For two points $p,q$ of a subset $\mathcal P$ of the plane, we say that segment $pq$ is a \emph{diameter} of $\mathcal P$ if no other segment with both endpoints in $\mathcal P$ has length greater than that of $pq$.
    It is well-known that, by the triangle inequality, any two diameters of a point set must intersect or share an endpoint. 

    For point $P$ and line $\ell$, we use $\dist(P, \ell)$ to denote the distance from $P$ to $\ell$.
    We will let $[A_1A_2A_3]$ denote the (unsigned) area of triangle $A_1A_2A_3$.
    For distinct points $A_1, A_2, A_3$ and $B_1, B_2, B_3$, we will use $A_1A_2A_3 \cong B_1B_2B_3$ to imply that there exists a rigid plane isometry mapping $A_i$ to $B_i$ for $i=1,2,3$.
    On the other hand, we will use $\{A_1A_2A_3\} \cong \{B_1B_2B_3\}$ to denote triangle congruence ignoring vertex order---that there exists a permutation $\sigma: \{1,2,3\} \to \{1,2,3\}$ such that $A_1A_2A_3 \cong B_{\sigma(1)}B_{\sigma(2)}B_{\sigma(3)}$ for $i=1,2,3$. 
\end{notation}

We first prove a simple lemma that will be used repeatedly later on, characterizing congruent triangles sharing a side length.

\begin{lemma}\label{lem.case-splitter}
    If $A$, $B$, $C$, $D$ are four distinct points such that $\{ABC\} \cong \{ABD\}$, then either $ABC \cong ABD$ or $ABC \cong BAD$, and (at least) one of the following conditions are met:
    \begin{itemize}
        \item $D$ is the reflection of $C$ across $AB$;
        \item $D$ is the reflection of $C$ across the midpoint of $AB$;
        \item $D$ is the reflection of $C$ across the perpendicular bisector of $AB$.
    \end{itemize}
\end{lemma}

\begin{figure}[ht]
    \centering
    \begin{minipage}{0.5\textwidth}
        \centering
        \includegraphics[width=4cm]{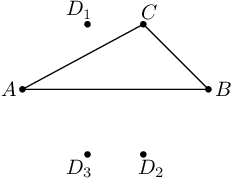}
        \caption{The possible positions of $D$ in \Cref{lem.case-splitter}.}
        \label{fig.case-splitter}
    \end{minipage}\begin{minipage}{0.5\textwidth}
        \centering
        \includegraphics[width=4cm]{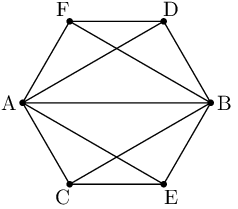}
        \caption{Labeling of hexagon vertices in proof of \Cref{lem.hexagon-regularizer}.}
        \label{fig.hexagon-regularizer}
    \end{minipage}
\end{figure}

\begin{proof}
    We first show that either $ABC \cong ABD$ or $ABC \cong BAD$.
    Casework on the true point order of the congruence:
    \begin{itemize}
        \item If $ABC \cong ABD$ or $ABC \cong BAD$, then we are already done.
        \item If $ABC \cong ADB$ then $AB=AD=AC$, so $ABC \cong ABD$.
        \item If $ABC \cong BDA$ then $AC=AB=BD$, so $ABC \cong BAD$.
        \item If $ABC \cong DAB$ then $AB=AD=BC$, so $ABC \cong BAD$.
        \item If $ABC \cong DBA$ then $AB=BD=BC$, so $ABC \cong ABD$.
    \end{itemize}

    When $ABC \cong ABD$, the only non-identity rigid isometry preserving $A$ and $B$ is the reflection over the line $AB$.
    On the other hand, when $ABC \cong BAD$, the only rigid isometries swapping $A$ and $B$ are the reflection across the midpoint of $AB$ (i.e., $180\dg$ rotation about the midpoint) and the reflection across the perpendicular bisector of $AB$. Hence $C$ and $D$ must be related via one of these three isometries.
\end{proof}

We now prove a highly useful lemma that, roughly speaking, enables us to convert information about distances and triangle congruences into a regular hexagon.

\begin{lemma}\label{lem.hexagon-regularizer}
    If $ACEBDF$ is a convex hexagon with vertices in that order (where consecutive vertices are allowed to be collinear) spanning at most three distinct triangles such that $AB$ is a diameter length and $\{ABC\} \cong \{ABD\} \cong \{ABE\} \cong \{ABF\}$, then $ACEBDF$ must be a regular hexagon.
\end{lemma}
\begin{proof}
    Note that $C$, $D$, $E$, $F$ must be related as dictated by \Cref{lem.case-splitter}, so $CEDF$ is a rectangle with center at the midpoint of $AB$, with $CE \parallel AB \parallel DF$.
    Note that $AD \parallel BC$ and $E$ is on the opposite side of $BC$ as $A$, so
    $\dist(E,AD) > \dist(B,AD)$. This implies $[ADE] > [ABD]$, so $\{ADE\} \not\cong \{ABD\}$.
    Also, $\dist(E,AD) > \dist(E,BC)$, so $[ADE] > [BCE]$ which implies $\{ADE\} \not\cong \{BCE\}$.
    Since $ABFD$ is an isosceles trapezoid with $DF < AB$, we also have $\{ABD\} \not\cong \{AFD\} \cong \{BCE\}$. This means the three distinct triangles must be $\{ADE\}, \{BCE\}$, and $\{ABD\}$.

    Now, $\{BDE\} \not\cong \{ADE\}$ because they are both isosceles, but their legs $BD \neq AD$. Also $\{BDE\} \not\cong \{ABD\}$ because neither $AD$ or $AB$ can be equal in length to $BE=BD$, This means $\{BED\} \cong \{BCE\}$. This implies $BE=CE$, and it follows by a symmetry argument that the hexagon is equilateral. It also implies that $\angle EBD = \angle CEB$, and it follows that the hexagon is equiangular, so it is regular.
\end{proof}

We are now ready to prove \Cref{thm.3-triangles}. Roughly speaking, we approach the proof by repeatedly identifying congruency classes of triangles, until we either derive a contradiction or arrive at the premises of \Cref{lem.hexagon-regularizer}.

\begin{proof}[Proof of \Cref{thm.3-triangles}]
    We prove that if six distinct points in the plane span at most three congruency classes of triangles, then these six points must be a regular hexagon.
    This implies \Cref{thm.3-triangles}, since for any seven points in the plane, we can find a subset of six points that do not form a regular hexagon, and therefore span at least four congruency classes of triangles.

    Label the six points $ABCDEF$ such that $AB$ is a diameter length.
    By the Pigeonhole principle, at least two of the triangles $ABC$, $ABD$, $ABE$, $ABF$ are congruent. Without loss of generality assume $\{ABC\} \cong \{ABD\}$; by \Cref{lem.case-splitter} we have three cases:
    either $ABCD$ is an isosceles trapezoid, with bases $AB, CD$ (and $AB > CD$),
    or $ABCD$ is a kite with perpendicular diagonals $AB$ and $CD$,
    or $ABCD$ is a parallelogram with diagonals $AB$ and $CD$.

    This can be further split into four cases, which we will deal with in turn:
    \begin{enumerate}
        \item $ACBD$ is a rhombus;
        \item $ACBD$ is a kite but not a rhombus;
        \item $ACBD$ is a parallelogram but not a rhombus;
        \item $ACBD$ is an isosceles trapezoid.
    \end{enumerate}

    \emph{Case 1:} $ACBD$ is a rhombus (possibly square):
    At least one of $E,F$ is not the center of the rhombus; without loss of generality, assume it is $E$.
    Note that $\dist(E,AB)$, $\dist(E,BC)$, $\dist(E,CD)$, $\dist(E,DA)$ cannot all be distinct, else the four triangles $EAB$, $EBC$, $ECD$, $EDA$ (which have identical bases $AB = BC = CD = DA$) have different areas and are therefore pairwise noncongruent. It follows that $E$ is either on line $AB$, line $CD$, one of the external angle bisectors of the rhombus, or one of the midlines of the rhombus (as shown in the subcases below).

    \begin{figure}[ht]
        \centering
        \begin{minipage}{0.33\textwidth}
            \centering
            \includegraphics[width=4cm]{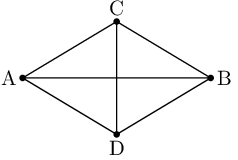}
            \caption{Case~1}
        \end{minipage}\begin{minipage}{0.33\textwidth}
            \centering
            \includegraphics[width=4cm]{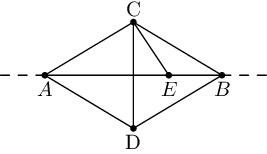}
            \caption{Subcase 1a}
        \end{minipage}\begin{minipage}{0.33\textwidth}
            \centering
            \includegraphics[width=4cm]{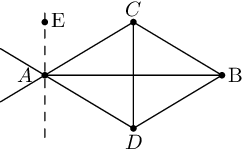}
            \caption{Subcase 1b}
        \end{minipage}
    \end{figure}

    \emph{Subcase 1a:} $E \in AB$ (and $E \in CD$ is analogous---no statements in this subcase rely on the fact that $AB$ is the diameter). We claim that triangles $\{ACE\}$, $\{CBE\}$, $\{ACD\}$, and $\{ABE\}$ are pairwise noncongruent:
    \begin{itemize}
        \item $\{ABE\}$ is not congruent to any of the other triangles because it is degenerate and the others are not.
        \item $\{CBE\} \not\cong \{ACE\}$ because if $E$ is not the center of $ACBD$ then $AEC \not\cong BEC$.
        \item $\{CBE\} \not\cong \{ACD\}$, because othwerise $CBE$ is isosceles. If $CE = EB$, then $CBE \cong CDA$, and $CD = BC = AC = AD$ implies the rhombus is composed of two equilateral triangles, so that $\angle CBE = 30\dg \neq \angle ADC = 60\dg$, contradiction; if $CB = CE$ then $E = A$, contradiction; if $CB = BE$ then its vertex angle is either $\tfrac{\angle CAD}2$ or $180\dg-\tfrac{\angle CAD}2$, while the vertex angle of $ACD$ is $\angle ACD$. Since they are equal, $\angle CAD=120\dg$, which implies the rhombus is composed of two equilateral triangles and $E=A$, contradiction.
        \item $\{ACE\} \not\cong \{ACD\}$ by the same reasoning as above.
    \end{itemize}
    Thus we have found four noncongruent triangles among the six points, contradiction.

    \emph{Subcase 1b:} $E$ lies on the external angle bisector of $A$ (and that of $B$ is analogous) is impossible because $BE > AB$, but $AB$ is a diameter.

    \begin{figure}[ht]
        \centering
        \begin{minipage}{0.5\textwidth}
            \centering
            \includegraphics[width=4cm]{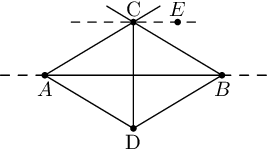}
            \caption{Subcase~1c}
        \end{minipage}\begin{minipage}{0.5\textwidth}
            \centering
            \includegraphics[width=4cm]{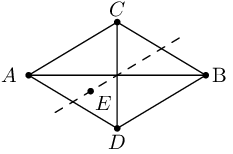}
            \caption{Subcase~1d}
        \end{minipage}
        \label{fig.case-1}
    \end{figure}

    \emph{Subcase 1c:} $E$ lies on the external angle bisector of $C$ (and that of $D$ is analogous). We may assume $CD < AB$, else this reduces to the previous case. Without loss of generality, assume $E$ and $B$ are on the same side of line $CD$.
    We claim $\{ABC\}$, $\{ACD\}$, $\{ECD\}$, $\{EAC\}$ are pairwise noncongruent:
    \begin{itemize}
        \item $\{ABC\}$, $\{ACD\}$. and $\{ECD\}$ are pairwise noncongruent because they are obtuse, acute, and right, respectively.
        \item $\{EAC\} \not\cong \{ABC\}$ because both are obtuse, while the obtuse angle $\angle ECA$ is greater than the obtuse angle $\angle ACB$.
        \item $\{EAC\} \not\cong \{ACD\}$ or $\{ECD\}$ because $\{EAC\}$ is obtuse, while $\{ACD\}$ and $\{ECD\}$ are acute and right, respectively.
    \end{itemize}
    Thus we have found four noncongruent triangles among the six points, contradiction.

    \emph{Subcase 1d:} $E$ lies on the midline parallel to $AC$ (and the midline parallel to $BC$ follows analogously). We claim $\{EBC\}$, $\{EDA\}$, $\{EAC\}$, $\{ABC\}$ are pairwise noncongruent:
    \begin{itemize}
        \item $\{EBC\} \not\cong \{EDA\}$ because $E$ is not the center of the rhombus, hence $\dist(E,BC) \neq \dist(E,AD)$, so the two triangles have different areas.
        \item $\{EBC\} \not\cong \{EAC\}$ because otherwise, by \Cref{lem.case-splitter} either $EBC \cong EAC$, implying $EB = EA$ and forcing $E$ to be the center of the rhombus, or $EBC \cong CAE$, implying $BC = AE$ and $AC = BE$, which is impossible.
        \item $\{EDA\} \not\cong \{EAC\}$ by symmetry with above.
        \item $\{EBC\} \not\cong \{ABC\}$ because by \Cref{lem.case-splitter} the only points $P$ such that $\{PBC\} \cong \{ABC\}$ are the reflection of $A$ over $BC$, which is not $E$ because $\angle ACB \geq 90\dg$, so it lies on the opposite side of line $AC$ from $E$; the reflection of $A$ over the midpoint of $BC$, which is not $E$ because it lies on line $BD$; and the reflection of $A$ over the perpendicular bisector of $BC$, which is not $E$ because it lies on line $AD$ and cannot be the midpoint of $AD$.
        \item $\{EDA\} \not\cong \{ABC\}$ by symmetry with above.
        \item $\{EAC\} \not\cong \{ABC\}$ because $\dist(E,AC) = \frac12 \dist(B,AC)$, so the two triangles have different areas.
    \end{itemize}
    Thus we have found four noncongruent triangles among the six points, contradiction.

    This completes all cases when $ACBD$ is a rhombus.

    \emph{Case 2:} $ACBD$ is a kite but not a rhombus. Then $\{ABC\}$, $\{ACD\}$, and $\{BCD\}$ are pairwise noncongruent, which implies these are the only three congruency classes of triangles.
    Without loss of generality assume $AC < BC$.

    \begin{figure}[ht]
        \centering
        \begin{minipage}{0.5\textwidth}
            \centering
            \includegraphics[width=4cm]{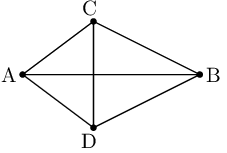}
            \caption{Subcase~2a}
        \end{minipage}\begin{minipage}{0.5\textwidth}
            \centering
            \includegraphics[width=4cm]{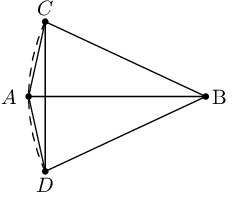}
            \caption{Subcase~2b}
        \end{minipage}
    \end{figure}

    \emph{Subcase 2a:} $CD \neq AB$ and $BC \neq AB$. Then there is only one congruency class of triangle possessing a side of length $AB$, so $\{EAB\} \cong \{FAB\} \cong \{ABC\}$, and \Cref{lem.hexagon-regularizer} implies that the six points form a regular hexagon.

    \emph{Subcase 2b:} $CD \neq AB$ and $BC = AB$. Then $A,C,D$ are equally-spaced on a circle with center $B$. The only two congruency classes of triangles possessing a side of length $AB$ are $\{ABC\}$ and $\{BCD\}$. If triangles $\{EAB\}$ and $\{FAB\}$ are congruent to $\{ABC\}$ then we are again done by \Cref{lem.hexagon-regularizer} (and this is in fact impossible, because here $BC \neq AB$). Without loss of generality assume that it is $\{ABE\}$ which is noncongruent to $\{ABC\}$, so $\{ABE\} \cong \{BCD\}$.
    Since $CD < AB$ we have $\angle CBD < 60\dg$, so $\angle CBA = \angle ABD < 30\dg$; it follows that $AC < CD < AB = BC$.
    If $ABE \cong CBD$ then without loss of generality assume that $C,A,D,E$ lie on the circle with center $B$ in that order; then $\{BCE\}$ is a fourth distinct triangle.
    On the other hand, if $ABE \cong BCD$, then without loss of generality assume $E$ is on the same side of $AB$ as $C$. Now $\{EBD\}$ is a fourth distinct triangle, because
    the only congruency class of triangle having sides of length $BE$ and $BD$ is $BCD$, and clearly $\{BCD\} \not\cong \{EBD\}$ since the latter is not isosceles and the former is.

    \begin{figure}[ht]
        \centering
        \begin{minipage}{0.5\textwidth}
            \centering
            \includegraphics[width=4cm]{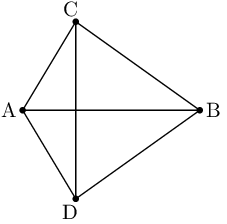}
            \caption{Subcase~2c}
        \end{minipage}\begin{minipage}{0.5\textwidth}
            \centering
            \includegraphics[width=4cm]{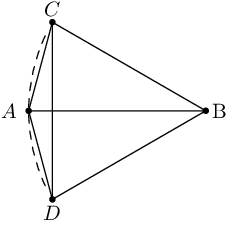}
            \caption{Subcase~2d}
        \end{minipage}
    \end{figure}

    \emph{Subcase 2c:} $CD = AB$ and $BC \neq AB$. If both triangles $\{EAB\}$ and $\{FAB\}$ are congruent to $\{ABC\}$ then we are again done by \Cref{lem.hexagon-regularizer}; else without loss of generality assume it is $\{ABE\}$ which is not congruent to $\{ABC\}$. This means either $\{EAB\} \cong \{ACD\}$ or $\{EAB\} \cong \{BCD\}$, which implies either $EAB \cong ACD$ or $EAB \cong BCD$ (since both triangles $ACD$ and $BCD$ are isosceles with base length $AB$).
    Without loss of generality, assume $E$ is on the same side of $AB$ as $C$.
    If $EAB \cong ACD$ then $EA=AC$. The only isosceles triangle with leg length $AC$ is $ACD$, but $\{EAC\} \not\cong \{ACD\}$ since their vertex angles $\angle CAE$ and $\angle CAD$ are unequal, contradiction.
    If $EAB \cong BCD$ then $EB=BC$. The only isosceles triangle with leg length $BC$ is $BCD$, but $\{EBC\} \not\cong \{BCD\}$, contradiction.

    \emph{Subcase 2d:} $CD = AB$ and $BC = AB$.
    Then $ACBD$ must be as shown, with $BCD$ equilateral and $A$ the arc midpoint of $CD$ on the circle with center $B$. Now $\{BCD\}$, $\{ACD\}$, $\{ABC\}$ are noncongruent, so $\{EAB\}$ must be congruent to one of these three triangles. We perform the finite case-check and verify that in all cases a fourth distinct triangle is produced.

    This completes all cases where $ACBD$ is a kite.

    \emph{Case 3:} $ACBD$ is a parallelogram with diagonals $AB$ and $CD$. We may assume it is not a rhombus, since otherwise we reduce to Case 1. We may also assume that no four-vertex subset of the six vertices forms a kite, since otherwise we reduce to Case 2.

    \begin{figure}[ht]
        \centering
        \begin{minipage}{0.33\textwidth}
            \centering
            \includegraphics[width=4cm]{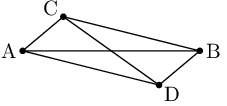}
            \caption{Case~3}
        \end{minipage}\begin{minipage}{0.33\textwidth}
            \centering
            \includegraphics[width=4cm]{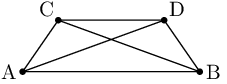}
            \caption{Subcase~4a}
        \end{minipage}\begin{minipage}{0.33\textwidth}
            \centering
            \includegraphics[width=4cm]{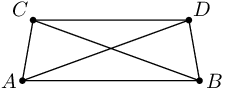}
            \caption{Subcase~4b}
        \end{minipage}
    \end{figure}

    Now, $\{EAC\} \not\cong \{EBC\}$, else some four points determine a kite by \Cref{lem.case-splitter} and we reduce to a previous case.
    Repeating this argument, $\{EAC\} \not\cong \{EBC\} \not\cong \{EBD\} \not\cong \{EAD\} \not\cong \{EAC\}$. However, since there are only three distinct triangles, some pair of these four triangles must be congruent; without loss of generality, assume the pair is $\{EAC\} \cong \{EBD\}$, which implies their areas are equal, so $E$ is on the midline of the parallelogram parallel to $AC$.
    Now if $ACBD$ is not a rectangle, then it is impossible for $\{EAC\}$ and $\{EBD\}$ to be congruent.
    So $ACBD$ is a rectangle. Note $\{ACD\} \not\cong \{ACE\}$ by area considerations, and $\{BCE\} \not\cong \{ACD\}$ since the former is isosceles and the latter is not.
    We already have $\{ACE\} \not\cong \{BCE\}$; by applying \Cref{lem.case-splitter} we also find $\{ADE\} \not\cong \{ACD\}$ and $\{ADE\} \not\cong \{ACE\}$.
    Since there are only three distinct triangles, we see $\{BCE\} \cong \{ADE\}$, which proves $E$ is the center. But then $EAC, EBC, EAB, ACD$ are pairwise noncongruent, contradiction.

    \emph{Case 4:} $ACDB$ is an isosceles trapezoid with bases $AB$ and $CD$. Without loss of generality, assume $C$ is closer to $A$ than $B$.
    If either $E$ or $F$ are the reflections of $C$ or $D$ across $AB$, then we reduce to one of the previous cases, so assume this is not the case.
    Observe that $\{ABC\} \not\cong \{ACD\}$.

    \emph{Subcase 4a:} If $BC = AD < AB$ then $\{EAB\} \not\cong \{ABC\}$, Otherwise, $E$ is either $D$, the reflection of $C$ over $AB$, or the reflection of $C$ over the midpoint of $AB$; all of these are disallowed (the latter two produce kites $ACBE$ and $ADBE$, respectively).
    Also, $\{EAB\} \not\cong \{ACD\}$, since all side lengths of the latter are shorter than $AB$. Therefore our three congruency classes are $\{EAB\}$, $\{ABC\}$, and $\{ACD\}$.
    By the same reasoning $\{FAB\} \not\cong \{ABC\}, \{ACD\}$, so $\{EAB\} \cong \{FAB\}$. If $ABEF$ is a rhombus, a non-rhombus kite, or a non-rhombus parallelogram, we reduce to Cases 1, 2, and 3 respectively, so by \Cref{lem.case-splitter} $ABEF$ is also a trapezoid, implying $\{AEF\} \not\cong \{ABE\}$.
    Also, $\{AEF\} \not\cong \{ABC\}$, else one of $AE, AF = AB$; without loss of generality assume $AE = AB$, so $AEB$ is an isosceles triangle with leg length $AB$ which is not congruent to any of the three previous triangles, giving us four distinct triangles. Thus $\{AEF\} \cong \{ACD\}$. Since $\angle ACD$ and $\angle AEF$ are both obtuse, they must correspond with each other, so $AD = AF$ and triangle $AFD$ is isosceles. Since $\{ABC\}$ and $\{ABE\}$ are not isosceles, we find $\{AFD\} = \{ACD\}$ so triangle $ACD$ is also isosceles. As $\angle ACD$ is obtuse, we must have $AC = CD$ and $FAD \cong ACD$, which implies $AC = CD = AD$. But this means $ACD$ is equilateral, implying $\angle ACD = 60\dg < 90\dg$, contradiction.

    \emph{Subcase 4b:} $BC = AD = AB$. Note that $\{EAB\}$ and $\{FAB\}$ are both noncongruent to $\{ABC\}$, since otherwise would produce a kite (by the same reasoning as in the previous subcase).

    If $\{EAB\}$ and $\{FAB\}$ are both noncongruent to $\{ACD\}$, then the three congruency classes must be $\{ABC\}$, $\{ACD\}$, and $\{EAB\} \cong \{FAB\}$. Since $\{EAB\}, \{FAB\} \not\cong \{ACD\}$ it follows that the lengths $EA=FB$ and $EB=FA$ are not equal to length $AB$, which reduces $ABEF$ to the previous subcase. Otherwise, without loss of generality, assume $\{EAB\} \cong \{ACD\}$, and that the orientation of the congruence is $EAB \cong CAD$.

    \begin{figure}[ht]
        \centering
        \begin{minipage}{0.5\textwidth}
            \centering
            \includegraphics[width=4cm]{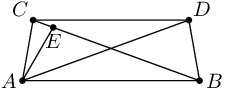}
            \caption{Subsubcase~4b(i)}
        \end{minipage}\begin{minipage}{0.5\textwidth}
            \centering
            \includegraphics[width=4cm]{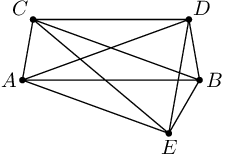}
            \caption{Subsubcase~4b(ii)}
        \end{minipage}
    \end{figure}

    \emph{Subsubcase 4b(i):} $E$ is on the same side of $AB$ as $C$ and $D$. Then $\angle EBA = \angle CDA = \angle CBA$ implies $C,B,E$ colinear, at which point the four triangles $\{BCE\}$, $\{ACE\}$, $\{ABC\}$, $\{ACD\}$ are pairwise noncongruent: $\{BCE\}$ is degenerate, $\{ACE\}$ and $\{ABC\}$ are acute isosceles triangles with different leg lengths, and $\{ACD\}$ is obtuse. We have found four triangles, contradiction.

    \emph{Subsubcase 4b(ii):} $E$ is on the opposite side of $AB$ as $C$ and $D$. Without loss of generality assume $BEA \cong BDC$; note that $EB=BD=AC$, $EA=CD$, $AB=BC=AD=CE$, and the five points are symmetric across the perpendicular bisector of $BC$.

    If $BE \neq AE$, then all lengths of triangle $BDE$ are are less than $AB$, so $\{BDE\}$ is not congruent to $\{ABC\}$ or $\{ABD\}$.
    Also, $\{DEC\}$ has no side length equal to $BD$, so $\{DEC\}$ is not congruent to $\{ABC\}$, $\{ABD\}$, or $\{BDE\}$, because it has no side length equal to $AC = BD$. This gives us four distinct triangles.

    If $BE = AE$, we have $AC=CD=DB=BE=EA$, so the pentagon is equilateral, and $DA=AB=BC=CE$. Therefore $\angle ABD = \angle BDA = \angle BCA = \angle BAC = \angle BEC$ and $\angle EBA = \angle ADC = \angle BCD = \angle BAE = \angle AEC$, which means $\angle EBD = \angle BDC = \angle DCA = \angle CAE = \angle AEB$, and the pentagon is also equiangular. This implies it is a regular pentagon.
    It is straightforward to check that one cannot add a sixth point to a regular pentagon without introducing two new distinct triangles.




    So we have exhausted all cases and are done.
\end{proof}

\section{Proof of \texorpdfstring{\Cref{thm.triangle-grid}}{Theorem 1.2}}
For ease of notation, we consider an $N \times N$ grid, so we wish to show that the number of triangles is between  $0.1558 N^4 + o(N^4)$ and $0.1875 N^4 + o(N^4)$ distinct triangles.

Let $[N] = \{0,1,\dots,N-1\}$, so the square lattice grid is $[N]\times[N]$.
Note that any triangle contained in this grid has at least one vertex as a vertex of its bounding box, so it can be mapped with suitable reflections, rotations, and translations onto a congruent triangle also contained in the grid with one vertex at the origin (see \Cref{fig.origin-shift}). This means it suffices to count the number of distinct triangles having one vertex at the origin.

\begin{figure}
    \centering
    \includegraphics[width=4cm]{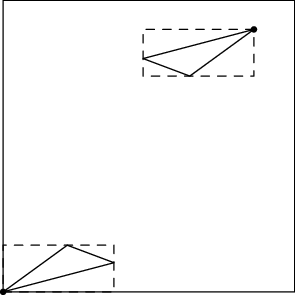}
    \caption{Lattice triangle with a congruent triangle having vertex at the origin}
    \label{fig.origin-shift}
\end{figure}

Let $O$ be the origin, $A = (a_1, a_2)$, and $B = (b_1, b_2)$.
The set of all triangles (including $OAB$) contained in $[N]\times[N]$ congruent to $OAB$ will be referred to as the \emph{congruency class} of $OAB$.
If $OAB$ is scalene and has no sides parallel to the coordinate axes, define the \emph{minimal congruency set} of $OAB$ as follows:
\begin{enumerate}
    \item If $OAB$ has two vertices on its bounding rectangle (without loss of generality assume $a_1 > b_1$ and $a_2 > b_2$), the minimal congruency set consists of the four triangles $\{(0,0),(a_1,a_2),(b_1,b_2)\}$, $\{(0,0),(a_2,a_1),(b_2,b_1)\}$, $\{(0,0),(a_1,a_2),(a_1-b_1,a_2-b_2)\}$, and $\{(0,0),(a_2,a_1),(a_2-b_2,a_1-b_1)\}$.
          This is shown in \Cref{fig.minimal-type-1}.
    \item If $OAB$ has all three vertices on its bounding rectangle, the minimal congruency set consists of the two triangles $\{(0,0),(a_1,a_2),(b_1,b_2)\}$ and $\{(0,0),(a_2,a_1),(b_2,b_1)\}$.
          This is shown in \Cref{fig.minimal-type-2}.
\end{enumerate}
Call the congruency class of $OAB$ \emph{minimal} if it contains only the triangles in the minimal congruency set.

\begin{figure}[ht]
    \centering
    \includegraphics[width=10cm]{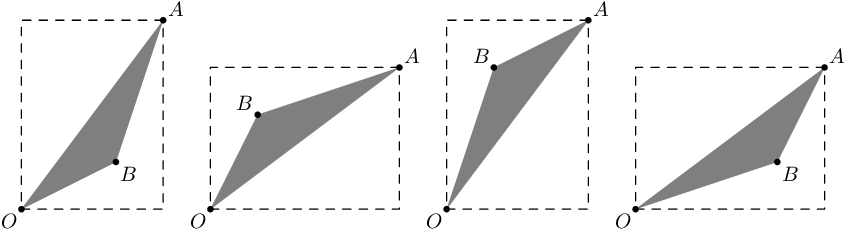}
    \caption{First type of minimal congruency class}
    \label{fig.minimal-type-1}
\end{figure}

\begin{figure}[ht]
    \centering
    \includegraphics[width=8cm]{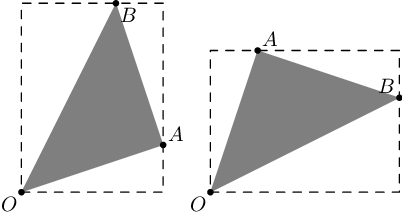}
    \caption{Second type of minimal congruency class}
    \label{fig.minimal-type-2}
\end{figure}

For angle $\theta$ and lattice point $P$, we say that $P$ is \emph{rotatable by} $\theta$ if the rotation of $P$ counterclockwise about the origin by $\theta$ is also a lattice point.
Call $P$ \emph{rotatable} if there exists $\theta \not\in\{0\dg,90\dg,180\dg,270\dg\}$ such that $P$ is rotatable by $\theta$.
Note that $P$ is rotatable by $\theta$ if and only if it is rotatable by $\theta+90\dg$, so if it is rotatable then there exists $0<\theta<90\dg$ such that $P$ is rotatable by $\theta$.

Similarly, for triangle $OAB$ and angle $\theta$, we say that $OAB$ is \emph{rotatable by} $\theta$ if $A$ and $B$ are both rotatable by $\theta$, and that $OAB$ is \emph{rotatable} if there exists $\theta \not\in\{0\dg,90\dg,180\dg,270\dg\}$ such that $A$ and $B$ are both rotatable by $\theta$.

We now show that rotatability of a triangle is linked to minimality of its congruency class:

\begin{lemma}\label{lem.rotatable-minimal}
    If a scalene non-right triangle with one vertex at the origin is not rotatable, then its congruency class is minimal.
\end{lemma}
\begin{proof}
    Suppose for the sake of contradiction there existed a triangle $OAB$ (in the first quadrant) which is not rotatable but also not minimal; by the latter assumption there exists another triangle $OCD$ (also in the first quadrant) such that $\{OCD\} \cong \{OAB\}$ and $OCD$ is not in $OAB$'s minimal congruency set. If triangles $\{OCD\}$ and $\{OAB\}$ are not oriented identically (i.e. the isometry mapping one to the other is a rotation rather than a glide reflection), then reflect $OCD$ over the line $y=x$ to make it oriented identically to $OAB$ while keeping it in the first quadrant.

    Consider the point order of the congruence.
    If $OCD \cong OAB$, then $OCD$ must be $OAB$, since otherwise would contradict rotatability of $OAB$. However, this contradicts the assumption that $OCD$ is not one of the triangles in $OAB$'s minimal congruency set. Similarly $ODC \cong OAB$ leads to contradiction.

    \begin{figure}[h]
        \centering
        \includegraphics[width=10cm]{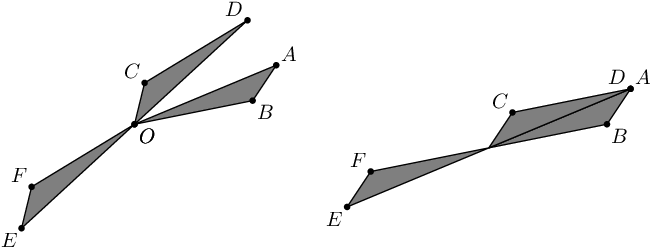}
        \caption{Positioning of triangles $OAB$, $OCD$, and $OEF$.}
        \label{fig.triangle-rotation}
    \end{figure}

    Otherwise, without loss of generality assume that $DOC \cong OAB$, then translate $DOC$ to $OEF$ as illustrated in \Cref{fig.triangle-rotation}, so $OEF \cong OAB$ implies, by the non-rotatability assumption, that $OEF$ is a rotation of $OAB$ by an integer multiple of $90\dg$.
    The translation sends $D$ to $O$ and $O$ to $E$, so $D$ and $E$ are reflections over $O$. The only rotation of $OAB$ by an integer multiple of $90\dg$ to $OEF$ such that the reflection of $E$ over $O$ is in the first quadrant is the rotation by $180\dg$, so $OEF$ is the $180\dg$ rotation of $OAB$, This implies $D=A$ and $C=A-B$ (denoting vector subtraction), which again contradicts the assumption that $OCD$ is not one of the triangles in $OAB$'s minimal congruency set.
\end{proof}

Note that if a non-origin lattice point $P$ is rotatable by $0\dg<\theta<90\dg$, then $\cos\theta$ and $\sin\theta$ must be rational, so can be expressed as $\frac pr$ and $\frac qr$ for primitive Pythagorean triple $(p,q,r)$.

\begin{lemma}\label{lem.rotatable-points-easy}
    Let $(p,q,r)$ be a primitive Pythagorean triple, and let $\theta = \arccos(q/r) = \arcsin(p/r)$.
    The number of points in $[N]\times[N]$ rotatable by $\theta$ is at most
    \[\begin{cases}
            0                    & \text{if $2N^2 < r$},           \\
            N                    & \text{if $N \leq r \leq 2N^2$}, \\
            r \lceil N/r\rceil^2 & \text{if $r < N$}.
        \end{cases}\]
\end{lemma}
\begin{proof}
    If $r > 2N^2$ then no points are rotatable by $\theta$, since if $(a,b) \in [N]\times[N]$ rotates to $(c,d)$ then $a^2+b^2=c^2+d^2$ and $\theta = \arctan\left(\frac{ad-bc}{ac+bd}\right)$ (see \Cref{fig.arctan-difference})
    hence $\frac pq = \frac{ad-bc}{ac+bd}$ so $r \leq \sqrt{(ad-bc)^2 + (ac+bd)^2} = \sqrt{(a^2+b^2)(c^2+d^2)} = a^2 + b^2 < 2N^2$, contradiction.

    \begin{figure}
        \centering
        \includegraphics[width=4cm]{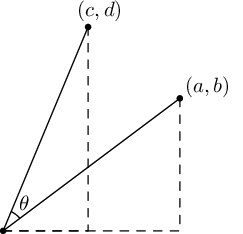}
        \caption{$(a,b)$ is rotatable by angle $\theta$, and $\tan\theta = \frac{\tfrac dc - \tfrac ba}{1 + \tfrac dc \cdot \tfrac ba} = \frac{ad-bc}{ac+bd}$.}
        \label{fig.arctan-difference}
    \end{figure}

    In the remaining cases, it will be useful to note that $(a,b)$ rotates by $\theta$ to $\left(\frac{ap-bq}r, \frac{aq+bp}r\right)$, which means $(a,b)$ is rotatable by $\theta$ if and only if $ap \equiv bq \pmod{r}$.

    The above implies that there is exactly one point rotatable by $\theta$ in any $1\times r$ subrow of the lattice.
    If $N \leq r$ then one can cover $[N]\times[N]$ in $N$ such subrows, so the number of lattice points rotatable by $\theta$ in $[N]\times[N]$ is at most $N$.
    If $r < N$ then one can cover $[N]\times[N]$ in $\lceil N/r \rceil^2$ squares of size $r$, each having $r$ points rotatable by $\theta$, so the number of points rotatable by $\theta$ is at most $r \lceil N/r \rceil^2$.
\end{proof}

We now present a refinement of the middle case of \Cref{lem.rotatable-points-easy} that holds when $r$ is sufficiently large.
\begin{lemma}\label{lem.rotatable-points-hard}
    Let $M>4$ be an arbitrarily chosen positive integer,
    let $(p,q,r)$ be a primitive Pythagorean triple, and let $\theta = \arccos(\tfrac qr) = \arcsin(\tfrac pr)$.
    If $r \geq 2M^4 N$ and $N \geq M^5$, the number of points in $[N]\times[N]$ rotatable by $\theta$ is at most $N/M$.
\end{lemma}
\begin{proof}
    From the observations in the proof of \Cref{lem.rotatable-points-easy},
    letting $c = p^{-1}q \pmod r$, rotatability by $\theta$ can be expressed as $a \equiv bc \pmod r$.
    Hence rotatability by $\theta$ is a linear relation, i.e. if points $S$ and $T$ are rotatable by $\theta$ then so are $uS + vT$ for any $u,v\in\mathbb Z$.
    Note that the points rotatable by $\theta$ in $[N]\times[N]$ are in the form of $(cb \bmod r, b)$ for $0\leq b < N$ and $0\leq cb\bmod r < N$.

    Since $r>N$, there is at most one point rotatable by $\theta$ per row of $[N]\times[N]$ by the above reasoning.
    Suppose there were more than $N/M$ points rotatable by $\theta$ contained in $[N]\times[N]$; then by the Pigeonhole Principle there exist $0\leq b_1 < b_2 < N$ such that $b_2 - b_1 < M$, and the rows $b_1, b_2$ both contain points rotatable by $\theta$ in $[N]\times[N]$.
    Let the vector connecting these points be $(u,v)$ with $v>0$, so we have $0<v<M$. Meanwhile by \Cref{lem.rotatable-points-easy}, since $r\geq 2M^4N > 2M^4$, no point $(a,b)$ with $a,b\leq M^2$ is rotatable by $\theta$ so $M^2 < |u| \leq N$.
    Note that $\left|\frac vu\right| < \frac{M}{M^2} = \frac1M$.

    For any $b$, consider the rows $b, b+v, \dots, b+2M^4 \left\lfloor \frac{N}{|u|} \right\rfloor v$.
    The corresponding columns for rotatable points by $\theta$ in $[r]\times[r]$ in those rows are, respectively, $\left\{a, a+u, \dots, a+2M^4\left\lfloor \frac{N}{|u|} \right\rfloor u\right\} \pmod r$ for $a \equiv cb \pmod r$.
    Note that the largest difference in the set of column indices above is $2M^4\left\lfloor \frac{N}{|u|} \right\rfloor |u| \leq 2M^4N \leq r$, so this set of lattice points ``wraps around'' the modulus $r$ at most once (see \Cref{fig.flat-vector}).

    \begin{figure}[h]
        \centering
        \includegraphics[width=10cm]{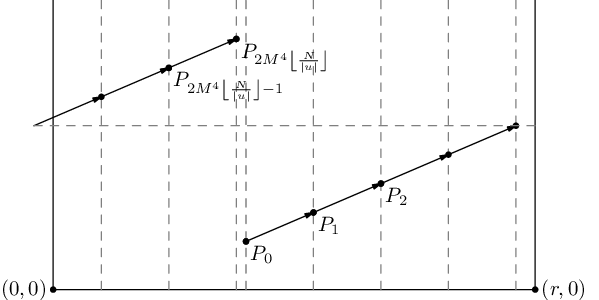}
        \caption{Rotatable points differing by a multiple of $(u,v)$, here indexed by $P_k = (a + ku \mod r,b+kv).$}
        \label{fig.flat-vector}
    \end{figure}

    After removing the last element from the set, the remaining corresponding columns
    \[\left\{a, a+u, \dots, a+\left(2M^4\left\lfloor \frac{N}{|u|} \right\rfloor -1\right) u\right\} \pmod r\]
    when reduced to residues modulo $r$ are all spaced a distance of at least $u$ from each other.
    Thus the square grid $[N]\times[N]$ can contain at most $\left\lceil \frac N{|u|} \right\rceil$ points rotatable by $\theta$ from this set, by counting columns.

    Partition the rows of $[N]\times[N]$ into sets of the form $\left\{k, k+v, \dots, k+\left(2M^4\left\lfloor \frac{N}{|u|} \right\rfloor-1\right) v\right\}$ such that each row is in at least one set.
    One can do this with at most $v \left\lceil \frac{\lceil N/v\rceil}{2M^4 \lfloor N/|u|\rfloor} \right\rceil$ sets. Since $|u| < N$, we have $\left\lfloor \frac{N}{|u|} \right\rfloor \geq 1$, so $\left\lfloor\frac{N}{|u|} \right\rfloor \geq \frac12 \cdot \frac{N}{|u|}$.
    This implies the number of sets is at most $\leq v\left(1 + \frac{1+\frac Nv}{\frac{M^4}{|u|}N}\right)$.

    We know each of these sets contains at most $\left\lceil\frac{N}{|u|}\right\rceil$ rows with points rotatable by $\theta$ in $[N]\times[N]$, so the number of points rotatable by $\theta$ in $[N]\times[N]$ is at most
    \[
        v\left(1 + \frac{1+\frac Nv}{\frac{M^4}{|u|}N}\right) \left\lceil\frac{N}{|u|}\right\rceil
        \leq
        v\left(1 + \frac{1+\frac Nv}{\frac{M^4}{|u|}N}\right) \left(1 + \frac{N}{|u|}\right).
    \]
    We now finish the calculation, recalling that $N\geq M^5$, $0<v<M$, and $M^2<|u|<N$:
    \begin{align*}
        v\left(1 + \frac{1+\frac Nv}{\frac{M^4}{|u|}N}\right) \left(1 + \frac{N}{|u|}\right)
         & = v\left(1 + \left(1+\frac{N}{|u|}\right)\left(\frac{1+\tfrac Nv}{\tfrac{M^4}{|u|}N}\right) + \frac{N}{|u|} \right)   \\
         & \leq v\left(1  + \frac{2N}{|u|} \frac{\frac{2N}v}{\frac{M^4}{|u|} N} + \frac N{M^2}\right)                            \\
         & \leq M + \frac{4N}{M^4} +v\left(\frac{N}{M^2}\right)                                                                  \\
         & \leq \frac{N}{M^3} + v\left(\frac{N}{M^2}\right) \leq  \frac{N}{M^3}  + (M-1)\left(\frac{N}{M^2}\right)\leq \frac NM,
    \end{align*}
\end{proof}

Let $\mathcal P$ denote the set of primitive Pythagorean triples, where triples $(p,q,r)$ and $(q,p,r)$ are counted separately.
\begin{lemma}\label{lem.rotatable-triangles}
    The number of rotatable triangles in $[N]\times[N]$ with one vertex at the origin is at most
    \[
        N^4 \left(\sum_{(p,q,r)\in\mathcal P} \frac1{2r^2} \right) + o(N^4)
        \leq
        0.0633 N^4 + o(N^4).
    \]
\end{lemma}
\begin{proof}
    Let $N>5^5$ and $M = \lfloor \sqrt[5]{N} \rfloor$.
    We begin by upper-bounding the number of rotatable triangles in $[N]\times[N]$. This may be done by finding the number of triangles which are rotatable by $\theta$ for each $\theta$, then summing over all $\theta$.
    Take $N$ large and $M\leq \sqrt N$; use $f(\theta)$ to denote the number of points in $[N]\times[N]$ rotatable by $\theta$.

    The number of rotatable triangles contained in $[N]\times[N]$ is at most
    \[
        \sum_{0\dg<\theta<90\dg} \dbinom{f(\theta)}2
        =
        \sum_{(p,q,r) \in \mathcal P} \dbinom{f(\arctan(q/p))}2.
    \]
    By \Cref{lem.rotatable-points-easy} and \Cref{lem.rotatable-points-hard},
    \[
        \sum_{0\dg<\theta<90\dg} \dbinom{f(\theta)}2 \leq
        \sum_{\substack{(p,q,r) \in \mathcal P \\ r<N}} \dbinom{r \lceil N/r \rceil^2}2
        +
        \sum_{\substack{(p,q,r) \in \mathcal P \\ N \leq r < 2M^4 N}} \binom N2
        +
        \sum_{\substack{(p,q,r) \in \mathcal P \\ 2M^4 N \leq r \leq 2N^2}} \binom {N/M}2.
    \]
    The first summand is at most $N^4 \sum_{(p,q,r)\in\mathcal P} \frac1{2r^2} + o(N^4)$.
    Since the number of primitive Pythagorean triples with hypotenuse less than $k$, counting $(p,q,r)$ and $(q,p,r)$ separately, is $\frac k\pi + o(k)$~\cite[327-328]{pythagorean_triples_counting},
    the second summand is $o(N^4)$ and
    the third summand is at most $\frac{N^4 }{\pi M^2} + o(N^4)$.

    Thus the number of rotatable triangles is at most
    \[
        N^4 \left(\sum_{(p,q,r)\in\mathcal P} \frac1{2r^2} + \frac{1}{\pi M^2}\right) + o(N^4).
    \]
    Since $M=\lfloor \sqrt[5]{N} \rfloor$, the $N^4 \frac1{\pi M^2}$ term is $o(N^4)$.
    Hence the number of rotatable triangles in $[N]\times[N]$ is at most
    \[
        N^4 \left(\sum_{(p,q,r)\in\mathcal P} \frac1{2r^2} \right) + o(N^4).
    \]

    It remains to give the numerical upper bound on $\sum_{(p,q,r)\in\mathcal P} \frac1{2r^2}$.
    It is well-known that any primitive Pythagorean triple $(p,q,r)$ can be expressed as $(m^2-n^2, 2mn, m^2+n^2)$ or $(2mn, m^2-n^2, m^2+n^2)$ for positive integers $m$ and $n$. Since $m^2+n^2 \leq r$, we have $n \leq \sqrt r$. This means the number of primitive Pythagorean triples with hypotenuse $r$ is at most $2\sqrt r$.
    The rest is a straightfoward computer-assisted computation:
    \begin{align*}
        \sum_{(p,q,r)\in\mathcal P} \frac1{2r^2}
         & =
        \sum_{\substack{(p,q,r)\in\mathcal P         \\ r \leq 10^5}} \frac1{2r^2} + \sum_{\substack{(p,q,r) \in \mathcal P \\ r > 10^5}} \frac1{2r^2} \\
         & \leq \sum_{\substack{(p,q,r)\in\mathcal P \\ r \leq 10^5}} \frac1{2r^2} + \sum_{r=10^5}^\infty \frac{\sqrt r}{r^2} \\
         & \leq \frac12 (0.1137) + 0.0064 < 0.0633.
    \end{align*}
    where the former summand is evaluated by enumerating the finite number of triples and summing, while the latter summand is evaluated by upper-bounding with a Riemann integral.
    (Here, the index cutoff of $10^5$ was arbitrarily chosen to facilitate the computation; the constant may be reduced slightly by increasing the index cutoff and employing more computational power.)
\end{proof}

We are now ready to prove \Cref{thm.triangle-grid}. It is equivalent to prove the following:

\begin{theorem}
    The number of distinct triangles in $[N]\times[N]$ is between $0.1558 N^4 + o(N^4)$ and $0.1875 N^4 + o(N^4)$.
\end{theorem}
\begin{proof}
    As noted before, it suffices to count the number of congruency classes of triangles with at least one vertex at the origin.

    The total number of such triangles is $\binom{N^2-1}2 = \frac12 N^4 + o(N^4)$.
    Since any line contains at most $N$ points inside $[N]\times[N]$, the number of right triangles and isosceles triangles are both $o(N^4)$, so the number of scalene non-right triangles is $\frac12 N^4 + o(N^4)$.

    \begin{figure}[h]
        \centering
        \includegraphics[width=4cm]{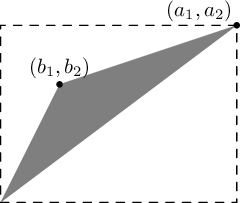}
        \hspace{0.1\textwidth}
        \includegraphics[width=5.25cm]{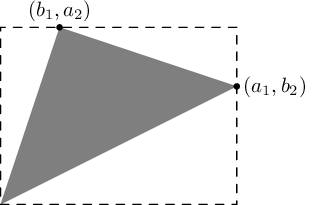}
        \caption{Pairing of triangles having two and three vertices on bounding box.}
        \label{fig.bounding-box-pair}
    \end{figure}

    Furthermore, consider pairing triangles $\{(0,0), (a_1,a_2), (b_1,b_2)\}$ and $\{(0,0), (a_,b_2), (a_2,b_1)\}$ (see \Cref{fig.bounding-box-pair}). This pairs the scalene non-right triangles with three vertices on its bounding box and the scalene non-right triangles with two vertices on its bounding box, except for $o(N^4)$ exceptions (where the pair of a scalene non-right triangle is paired with an isosceles or a right triangle). This implies that the number of scalene non-right triangles with three vertices on its bounding box and the number of scalene non-right triangles with two vertices on its bounding box are identical up to a constant difference of $o(N^4)$.
    Hence the number of each is $\frac14 N^4 + o(N^4)$.

    Let $A$ be the number of rotatable triangles in $[N]\times[N]$ having three vertices on its bounding box, and $B$ be the number of rotatable triangles $[N]\times[N]$ having two vertices on its bounding box. By \Cref{lem.rotatable-triangles}, $A+B \leq 0.0633 N^4 + o(N^4)$.

    The number of non-rotatable scalene non-right triangles having three vertices on its bounding box is therefore $0.25 N^4 - A + o(N^4)$, and the number of non-rotatable scalene non-right triangles having two vertices on its bounding box is $0.25 N^4 - B + o(N^4)$.

    By \Cref{lem.rotatable-minimal}, non-rotatable scalene non-right triangles have minimal congruency classes. The minimal congruency set of scalene non-right triangles with three vertices on its bounding box consists of exactly $2$ triangles, each with three vertices on their bounding boxes; the minimal congruency set of scalene non-right triangles with two vertices on its bounding box consists of exactly $4$ triangles, each with two vertices on their bounding boxes.
    Hence the number of congruency classes is at least
    \[
        \frac{0.25 N^4 - A + o(N^4)}2 + \frac{0.25 N^4 - B + o(N^4)}4
        \geq 0.1875 N^4 - \frac{A+B}2 + o(N^4),
    \]
    which is greater than $0.1558 N^4 + o(N^4)$ as desired.

    Meanwhile, since the congruency class of any scalene non-right triangle contains at least the triangles given by its minimal congruency set,
    we have the number of congruency classes total is at most
    \[
        \frac{0.25 N^4 + o(N^4)}2 + \frac{0.25 N^4 + o(N^4)}4 = 0.1875 N^4 + o(N^4).
    \]
\end{proof}

\section{Conjectures}

We present several natural conjectures based on this work.
Firstly, observing that the regular heptagon has $4$ distinct triangles, we conjecture that this is optimal:

\begin{conjecture}
    The maximum number of points in the plane spanning four distinct triangles is $7$, achieved only by the regular heptagon.
\end{conjecture}

In light of work done in $1$ and $2$ distinct triangles in higher dimensions~\cite{one_dist_triangle,two_dist_triangle}, we also suspect the maximum number of points in $\mathbb R^3$ spanning three distinct triangles may be $8$, e.g. is achievable by the vertices of a cube.

It would be interesting to determine the true constant $c$ for which the $n$-point square lattice determines $cn^2 + o(n^2)$ distinct triangles.

In addition, we have numerical evidence suggesting that the equilateral triangular lattice has more triangles than the square lattice: note that in the $N\times N$ triangular lattice, by a similar translational argument to the square lattice, it suffices to count triangles with at least one vertex on either a $60\dg$ or a $120\dg$ corner of the lattice. We used a computer program to count the number of distinct triangles $T_N$ in such a grid, and computed $\frac{T_N}{N^4}$ for $1 \leq N \leq 250$; a curvefit suggests that a triangular lattice of $n$ points spans approximately $0.2n^2$ distinct triangles.
Thus, we would like to also exclude the triangular lattice from optimal configurations.
We more strongly conjecture that no two-dimensional lattice has asymptotically fewer distinct triangles than the square lattice:
\begin{conjecture}
    Let $c_0$ be the real number such that the $n$-point square lattice spans $c_0 n^2 + o(n^2)$ distinct triangles.
    Then for any distinct nonzero vectors $a,b \in \mathbb R^2$, the $n$-point lattice generated by $a$ and $b$, i.e. $\left\{ua + vb \mid 0 \leq u,v < \sqrt n, u,v \in \mathbb Z\right\}$, spans at least $c_0 n^2 + o(n^2)$ distinct triangles.
\end{conjecture}

The investigation of non-rectangular lattice point sets could also be relevant to the discussion of this problem. For instance, one might consider the hexagonal grid formed by the triangular lattice, or sets of points formed by taking the intersection of a disk with any lattice. The authors believe this analysis will be more technically challenging, but may yield interesting results in the way of lower constant factors on the number of distinct triangles.

\section*{Acknowledgements}

This paper represents the results from a 2023 Polymath Jr. project, supported by NSF award DMS-2313292. We would like to thank Prof. Adam Sheffer and the Polymath Jr. staff for their support of the program; we are additionally grateful to Prof. Adam Sheffer for inspiring conversations about these problems. Palsson was supported in part by the Simons Foundation grant \#360560, and would also like to thank the Vietnam Institute for Advanced Study in Mathematics (VIASM) for the hospitality and for the excellent working condition, while part of this work was done.


\;

\end{document}